\documentclass{article}
\usepackage[utf8]{inputenc}
\usepackage[T2A]{fontenc}
\usepackage[english]{babel}

\usepackage{parskip}

\usepackage{amsmath}
\usepackage{amssymb}
\usepackage{amsthm}
\usepackage{amsfonts}
\usepackage{asymptote}
\usepackage{xcolor}
\usepackage{subcaption}
\usepackage{graphicx}
\usepackage{cancel}

\newtheorem{lemma}{Lemma}[]

\newtheorem{proposition}{Proposition}[]

\newtheorem{theorem}{Theorem}[]
\newtheorem{corollary}{Corollary}[]
\theoremstyle{definition}
\newtheorem{definition}{Definition}[]

\theoremstyle{plain}

\newcommand{\diam}{\operatorname{diam}}

\newcommand{\dis}{\operatorname{dis}}
\newcommand{\dist}{\operatorname{d}}

\renewcommand{\:}{\colon}

\renewcommand{\ss}{\subset}

\newcommand{\N}{\mathbb{N}}
\newcommand{\R}{\mathbb{R}}
\newcommand{\Z}{\mathbb{Z}}

\def\cC{{\cal C}}

\title{Geodesic in the Gromov--Hausdorff class for which the real line is an interior point}
\author{Ivan N. Mikhailov}
\date{}

\begin{document}
\maketitle

\begin{abstract}
In this note we construct a geodesic line in the Gromov--Hausdorff class for which the real line with a natural metric is an interior point. 
\end{abstract}



\section{Introduction}

In~\cite{GromovEng} M.\,Gromov introduced moduli spaces of the class of all metric spaces at finite Gromov--Hausdorff distances from a given metric space. It was mentioned that such moduli spaces are always complete and contractible (\cite{GromovEng}[section $3.11_{\frac{1}{2}_+}$]). In~\cite{BogatyTuzhilin} the authors suggested to work with such moduli spaces (they were called \emph{clouds}) in the sense of NBG set theory to avoid arising set-theoretic issues. While the completeness of each cloud was verified in~\cite{BogatyTuzhilin}, the contractibility of each cloud remains an open question for a number of reasons. The main issue here is that a natural homothety-mapping that takes a metric space $(X, d_X)$ into $(X, \lambda d_X)$ for some $\lambda > 0$ and generates a contraction of a cloud of all bounded metric spaces if $\lambda\to 0$, does not behave so well in case of unbounded metric spaces. Firstly, in~\cite{BogatyTuzhilin} it was shown that there exist metric spaces such that $d_{GH}(X, \lambda X) = \infty$ for some $\lambda > 0$. The simplest one is a geometric progression $X = \{3^n\:n\in\N\}$ with a natural metric, for which $d_{GH}(X, 2X) = \infty$. Secondly, even for clouds that are invariant under multiplication on all positive numbers a homothety-mapping may not be continuous. In~\cite{nglitghclitcotrl} it was shown that $d_{GH}(\Z^n, \lambda \Z^n)\ge \frac{1}{2}$ for all $\lambda > 1$, $n\in\N$.

In this note we continue the investigation of the geometry of the Gromov--Hausdorff class. We focus on the other open problem of constructing geodesics in the Gromov--Hausdorff class. In~\cite{Vihrov} a special class of metric spaces in a so-called general position was constructed, such that it is dense in the Gromov--Hausdorff class, and every two metric spaces from these class can be connected with a linear geodesic. However, it is still not known whether every two metric spaces at finite Gromov--Hausdorff distance from each other can be joint with a geodesic in the Gromov--Hausdorff class. Some new examples of geodesics lying in the cloud of the real line appeared in~\cite{nglitghclitcotrl}, and~\cite{IMT} (we review both these constructions in Section~\ref{section: clouds}). In this note, we construct the new geodesic line, for which the real line is an interior point. The existence of such geodesic is impossible in the cloud of bounded metric spaces, because of the ultrametric inequality from the first statement of Theorem~\ref{thm: boundedcloud} below.

\subsection*{Acknowledgements}

The author is grateful to his scientific advisor A.\,A.\,Tuzhilin and A.\,O.\,Ivanov for the constant attention to the work. 

Also the author is grateful to all the participants of the seminar <<Theory of extremal networks>>, leaded by A.\,O.\,Ivanov and A.\,A.\,Tuzhilin in Lomonosov Moscow State University for numerous fruitful discussions.

\section{Preliminaries}

\emph{A metric space} is an arbitrary pair $(X,\,d_X)$, where $X$ is an arbitrary set, $d_X\: X\times X\to [0,\,\infty)$ is some metric on it, that is, a nonnegative symmetric, positively definite function that satisfies the triangle inequality. 

For convenience, if it is clear in which metric space we are working, we denote the distance between points $x$ and~$y$ by~$|xy|$. Suppose $X$ is a metric space. By $U_r(a) =\{x\in X\colon |ax|<r\}$, $B_r(a) = \{x\in X\colon |ax|\le r\}$ we denote open and closed balls centered at the point~$a$ of the radius~$r$ in~$X$. For an arbitrary subset $A\subset X$ of a metric space~$X$, let $U_r(A) = \cup_{a\in A} U_r(a)$ be the open $r$-neighborhood of~$A$. For non-empty subsets $A \ss X$, $B \ss X$ we put $\dist(A,\,B)=\inf\bigl\{|ab|:\,a\in A,\,b\in B\bigl\}$.

\subsection{Hausdorff and Gromov--Hausdorff distances}

\begin{definition}
Let $A$ and $B$ be non-empty subsets of a metric space.
\emph{The Hausdorff distance} between $A$ and~$B$ is the quantity $$d_H(A,\,B) = \inf\bigl\{r > 0\colon A\subset U_r(B),\,B\subset U_r(A)\bigr\}.$$
\end{definition}

It is well-known that $d_H$ is a generalized pseudometric on the space~$\cC(X)$ of all non-empty closed subsets of a metric space~$X$. The wold <<generalized>> means that $d_H$ can be infinite between some pairs of closed subsets (for example, between a bounded and unbounded subsets).

\begin{definition} Let $X$ and $Y$ be metric spaces. The triple $(X', Y', Z)$, consisting of a metric space $Z$ and its two subsets $X'$ and $Y'$, isometric to $X$ and~$Y$, respectively, is called \emph{a realization of the pair} $(X, Y)$.
\end{definition}

\begin{definition} \emph{The Gromov-Hausdorff distance} $d_{GH} (X, Y)$ between $X$ and~$Y$ is the exact lower bound of the numbers $r\ge 0$ for which there exists a realization $(X', Y', Z)$ of the pair $(X, Y)$ such that $d_H(X',\,Y') \le r$. 
\end{definition}

Now let $X,\,Y$ be non-empty sets.  

\begin{definition} Each $\sigma\subset X\times Y$ is called a \textit{relation} between $X$ and~$Y$.
\end{definition}

By $\mathcal{P}_0(X,\,Y)$ we denote the set of all non-empty relations between $X$ and~$Y$.

We put $$\pi_X\colon X\times Y\rightarrow X,\;\pi_X(x,\,y) = x,$$ $$\pi_Y\colon X\times Y\rightarrow Y,\;\pi_Y(x,\,y) = y.$$ 

\begin{definition} A relation $R\subset X\times Y$ is called a \textit{correspondence}, if their restrictions $\pi_X|_R$ and $\pi_Y|_R$ are surjective.
\end{definition}

Let $\mathcal{R}(X,\,Y)$ be the set of all correspondences between $X$ and~$Y$.

\begin{definition} Let $X,\,Y$ be metric spaces, $\sigma \in \mathcal{P}_0(X,\,Y)$. The \textit{distortion} of $\sigma$ is the quantity $$\dis \sigma = \sup\Bigl\{\bigl||xx'|-|yy'|\bigr|\colon(x,\,y),\,(x',\,y')\in\sigma\Bigr\}.$$
\end{definition}

\begin{proposition}[\cite{BBI}]  \label{proposition: distGHformula}
For arbitrary metric spaces $X$ and~$Y$, the following equality holds $$2d_{GH}(X,\,Y) = \inf\bigl\{\dis\,R\colon R\in\mathcal{R}(X,\,Y)\bigr\}.$$
\end{proposition}

\subsection{Clouds}\label{section: clouds}

By~$\mathcal{VGH}$ we denote the class of all non-empty metric spaces, equipped with the Gromov--Hausdorff distance.

Note that~$\mathcal{VGH}$ is a proper class in the sense of the NBG set theory. In this theory all objects are \emph{classes} of one of the two following types: \emph{sets}, or \emph{proper classes}. A class is called a set if it belongs to some other class, and a proper class otherwise. It is important for us that for all classes the following natural constructions are defined: Cartesian product, maps between classes, metrics, pseudometrics, etc.  

\begin{theorem}[\cite{BBI}]
The Gromov--Hausdorff distance is a generalised pseudometric on~$\mathcal{VGH}$, vanishing on pairs of isometric metric spaces. Namely, the Gromov--Hausdorff distance is symmetric, satisfies a triangle inequality, though can vanish or be infinite between some pairs of non-isometric metric spaces.
\end{theorem}

A class~$\mathcal{GH}_0$ is obtained from~$\mathcal{VGH}$ by factorization over zero distances, i.e., over an equivalence relation $X\sim_0 Y$, iff $d_{GH}(X,\,Y) = 0$.

\begin{definition}
Consider an equivalence relation~$\sim_1$ on~$\mathcal{GH}_0$: $X\sim_1 Y$, iff~$d_{GH}(X,\,Y) < \infty$. We call the corresponding equivalence classes \emph{clouds}.
\end{definition}

For an arbitrary metric space~$X$, we denote a cloud containing~$X$ by~$[X]$. Let~$\Delta_1$ be a metric space, consisting of a single point. Hence, $[\Delta_1]$ is the cloud of all bounded metric spaces. 

Suppose that for some metric spaces $A$ and~$A'$, the equality holds~$d_{GH}(A, A') = 0$. Then, for arbitrary metric space~$B$, we also have~$d_{GH}(A, B) = d_{GH}(A', B)$. From this simple observation, it follows that all the results about~$d_{GH}(A, B)$ also hold if we exchange $A$ by~$A'$ such that~$d_{GH}(A, A') = 0$. Thus, instead of interpreting <<$A\in [X]$>> directly by definition so that $A$ is an equivalence class of all metric spaces on zero Gromov--Hausdorff distances from each other, we will mean that~$A$ is a certain member of this equivalence class. For example,~$X\in[\Delta_1]$ can be read as <<$X$ is a bounded metric space>> throughout the paper.

\begin{theorem}[\cite{BBI}] \label{thm: boundedcloud}
Let $X$ and~$Y$ be arbitrary bounded metric spaces. Then
\begin{itemize}
\item The inequalities hold 
$$\frac{1}{2}\bigl|\diam X - \diam Y\bigr|\le d_{GH}(X, Y)\le \max\bigl\{d_{GH}(X, \Delta_1), d_{GH}(Y, \Delta_1)\bigr\} = \frac{1}{2}\max\bigl\{\diam X, \diam Y\bigr\}.$$
\item A map $\Phi\:[\Delta_1]\times \R_{\ge 0}\to [\Delta_1]$, $\Phi(X, \lambda) = \lambda X$ is continuous and generates a contraction of the cloud~$[\Delta_1]$ if~$\lambda \to 0$.
\item A curve~$\lambda X$, $\lambda\in[0, +\infty)$ is a geodesic with respect to the Gromov--Hausdorff distance in the cloud~$[\Delta_1]$.
\end{itemize}   
\end{theorem}

By continuity of a map between clouds, we mean a continuity in a sense of a common definition of continuity in a point with respect to a metric. Namely, a map $f\: [A]\to  [B]$ is continuous in $x$ if, for arbitrary $\varepsilon > 0$, there exists $\delta > 0$ such that $f\bigl(U_\delta(x)\bigr)\subseteq U_\varepsilon(f(x))$, where $[A]$, and $[B]$ are equipped with the Gromov--Hausdorff distance.

We will also need two additional constructions of geodesics in the Gromov--Hausdorff class, lying in the cloud of the real line.

\begin{theorem}[\cite{nglitghclitcotrl}]\label{cor: simpleboundedgeodesicR}
For an arbitrary bounded metric space~$X$, a curve~$\R\times_{\ell^1}(tX)\: t\in[0,\,+\infty)$ is a geodesic in the Gromov--Hausdorff class such that $d_{GH}\bigl(\R\times_{\ell^1}(t_1X), \R\times_{\ell^2}(t_2X)\bigr) = \frac{|t_1-t_2|}{2}$.    
\end{theorem}

Let $X$ be an arbitrary boundedly compact, geodesic metric space. Consider its non-empty closed subsets $A$ and~$B$ such that $d=d_H(A,B)<\infty$. For $t\in[0,d]\ss\R$, we put $C_t=B_t(A)\cap B_{d-t}(B)$. Each $C_t$ is closed in $X$. In other words, $C_t\in\cC(X)$ for all~$t$.

\begin{theorem}[\cite{IMT}]\label{thm:HausdorffGeodesic}
In the introduced terminology the sets~$C_t$ are non-empty, and a curve~$t\mapsto C_t$ is a shortest curve in~$\cC(X)$.
\end{theorem}

A curve from Theorem~\ref{thm:HausdorffGeodesic} is called a \emph{canonical Hausdorff geodesic}.

Particularly, we need the following result

\begin{corollary}[\cite{IMT}, Corollary~$9.3$]\label{cor:Rgeodesic}
Let $X\ss\R$ be an arbitrary closed subset such that~$d_{GH}(X,\R)<\infty$. Then a canonical Hausdorff geodesic, connecting $X$ and~$\R$ is a shortest geodesic in the Gromov--Hausdorff class.
\end{corollary}

\section{Main theorem}

Let $X$ be an arbitrary bounded path-connected metric space of diameter~$1$. Fix $0 < \delta < \frac{1}{2}$. We put
\begin{align*}
&\Z_t = \cup_{n\in\Z}[n-t, n+t]\subset \R, \;t\in \Bigl[\frac{1}{2}-\delta, \frac{1}{2}\Bigr],\\
& \R_d = \R\times_{\ell^1}(dX),\;d\in [0, \delta],
\end{align*}
where $X\times_{\ell^1}Y$ is the Cartesian product $X\times Y$ equipped with the $\ell^1$-metric:
$$d_{X\times_{\ell^1}Y}\bigl((x, y), (x', y')\bigr) = d_X(x, x') + d_Y(y, y').$$

\begin{theorem}
By gluing $($and reparametrizing$)$ $\Z_t$, $t\in [\frac{1}{2}-\delta, \frac{1}{2}]$ and $\R_d$, $d\in [0, \delta]$, we obtain a shortest curve in the Gromov--Hausdorff class, for which $\R$ is an interior point.  
\end{theorem}

\begin{lemma}
The equality holds $d_{GH}(\R_{d_1}, \R_{d_2}) = \frac{|d_1-d_2|}{2}.$
\end{lemma}

\begin{proof}
If follows from the Theorem~\ref{cor: simpleboundedgeodesicR}.
\end{proof}

\begin{lemma}
The equality holds $d_{GH}(\Z_{t_1}, \Z_{t_2}) = |t_1-t_2|$.
\end{lemma}

\begin{proof}
It follows from the Corollary~\ref{cor:Rgeodesic}.
\end{proof}

Now note that it suffices to prove the following

\begin{lemma}\label{lemma: innerRmainlemma}
The equality $d_{GH}(\Z_t, \R_d) = \frac{d}{2} + \frac{1}{2} - t$.
\end{lemma}

\begin{proof}

By triangle inequality $$d_{GH}(\Z_t, \R_d)\le d_{GH}(\Z_t, \R) + d_{GH}(\R, \R_d) = \frac{d}{2}+\frac{1}{2}-t.$$ 

Let us prove the opposite inequality. 

Choose $R\in\mathcal{R}(\Z_t, \R_d)$. Suppose that $\dis R < d + 1 - 2t - \varepsilon$ for some~$\varepsilon > 0$. 

Construct a graph~$G$. To a segment $I_n = [n-t, n+t]$ in the space $\Z_t$ we match a vertex~$v_n$. We put $A_n = R(I_n)$. We connect vertices $v_n$ and~$v_m$ with en edge iff $\dist\bigl(A_n, A_m\bigr) = 0$. 

Suppose there exist adjacent $v_n$ and $v_m$ such that $|n-m| > 1$. Choose $x_k\in A_n$, $y_k\in A_m$, $k\in\N$ such that $|x_ky_k|\xrightarrow[k\to\infty]{}0$. Choose arbitrary points $p_k\in R^{-1}(x_k)\cap I_n$ and $q_k\in R^{-1}(y_k)\cap I_m$. Then $$\dis R \ge |p_kq_k| - |x_ky_k|\ge |n-m|-2t-|x_ky_k| \ge 1 + (1-2t)-|x_ky_k| > d + 1-2t,$$ where the last inequality holds for sufficiently large $k$ --- a contradiction.

\begin{lemma}\label{lemma: routeinG}
Graph $G$ is connected. 
\end{lemma}

\begin{proof}
Let $p = (t_1, x_1)\in A_n$, $q =(t_2, x_2)\in A_m$, and $\gamma\:[0,1]\to dX$ is continuous such that $\gamma(0) = x_1$, $\gamma(1) = x_2$. Then a curve $w(t) = \bigl(t\cdot t_1 + (1-t)\cdot t_2,\,\gamma(t)\bigr)\in \R_d$ is continuous, such that $w(0) = p$, $w(1) = q$. 

We put $W := w\bigl([0, 1]\bigr)\subset\R_d$. Since $w(t)$ is continuous, $\diam W < \infty$. Since $\dis R < \infty$, a subset $R^{-1}(W)$ is bounded. Hence, $R^{-1}(W)$ is covered by a finite union of segments $I_n$: $R^{-1}(W)\subset \cup_{i=1}^kI_{n_i}$. In particular, $W\subset \cup_{i=1}^k A_{n_i}$. Without loss of generality suppose that $A_{n_1} = A_n$, $A_{n_k} = A_m$.

Now we construct a path between $v_n$ and $v_m$ in~$G$. Put $t_1= \sup\bigl\{t\in[0,1]\: w(t)\in A_n\bigr\}$. If $t = 1$, then by construction of $G$ vertices $v_n$ and $v_m$ are adjacent. Suppose that $t < 1$. Since $w$ is continuous, $w(t+\frac{1}{n})\xrightarrow[n\to\infty]{} w(t)$ ($n$ is chosen in such way that $t+\frac{1}{n} < 1)$. However, each of the points $w(t+\frac{1}{n})$ belongs to $\cup_{i=2}^{k} A_{n_i}$. Thus, there exists a sequence $(a_n)_{n\in\N}$, $a_n\in\N$, tending to infinity, such that $w(t+\frac{1}{a_n})\in A_{n_j}$ for some~$j>1$. Hence, $v_n$ and $v_{n_j}$ are adjacent in~$G$, and $\sup\{t\in [0,1]\: w(t)\in A_{n_j}\} > \sup\{t\in[0, 1]\: w(t)\in A_n\}$. Since a collection $\{A_{n_i}\}_{i=1,\ldots,k}$ is finite, and the value $\sup\{t\in[0, 1]\: t\in A_n\}$ increases when we go from $A_n$ to $A_{n_j}$, by removing $A_n$ from $\{A_{n_i}\}_{i=1,\ldots,k}$ and repeating the same argument, we will construct a path from $v_n$ to $v_m$ in $G$ in a finite number of steps. 
\end{proof}

Since $v_n$ and $v_m$ are not adjacent in $G$ when $|n-m| > 1$, we conclude from Lemma~\ref{lemma: routeinG} that 
$v_n$ and $v_m$ are adjacent iff~$|n-m| = 1$.

\begin{lemma}\label{lemma: topologylemma}
For arbitrary $n\in\Z$ and $x\in (dX)$, we have $A_n\cap \bigl(\R\times\{x\}\bigr)\neq\emptyset$.
\end{lemma}

\begin{proof}
Arguing by contradiction, suppose there exist $x\in dX$ and $n\in\Z$ such that $A_n\cap \bigl(\R\times\{x\}\bigr)=\emptyset$. Fix $k = n-5 < n < l = n+5$. Choose arbitrary $(t', x')\in A_k$, $(t'', x'')\in A_l$. Choose continuous $\gamma, \gamma' \:[0,1]\to dX$ such that $\gamma(0) = x'$, $\gamma(1) = x$ and $\gamma'(0) = x$, $\gamma'(1) = x''$. Consider a gluing of three curves: $w_1(t) = \bigl(t', \gamma(t)\bigr)$, $t\in[0, 1]$, $w_2(\lambda) = \bigl(t'\cdot (1-\lambda) + t''\cdot \lambda,\, x\bigr)$, $\lambda\in [0, 1]$ и $w_3(s) = \bigl(t'', \gamma'(s)\bigr)$, $s\in[0,1]$. Since the union of the images of $w_1, w_2$, and $w_3$ is compact (denote it by~$K$), and $\dis R < \infty$, we conclude that preimage $R^{-1}(K)$ is bounded. Hence, $R^{-1}(K)$ has non-empty intersections only with a finite number of segments~$I_k$: $R^{-1}(K)\subset \cup_{i = 1}^m I_{n_i}$. In other words, $K\subset \cup_{i = 1}^m A_{n_i}$. Similarly to Lemma~\ref{lemma: routeinG} from this covering we can construct a path in~$G$ from $v_k$ to~$v_l$. However, note that $R^{-1}(K)\cap I_n=\emptyset$. Indeed, $R^{-1}\Bigl(w_2\bigl([0, 1]\bigr)\Bigr)\cap I_n = \emptyset$ by the condition of Lemma~\ref{lemma: topologylemma}. While for arbitrary $p\in w_1\bigl([0,1]\bigr)$ and $q\in R(I_n)$ the inequalities hold 
\begin{multline*}
|pq|\ge |w_1(0)q| - |w_1(0)p| \ge \dist(I_k, I_n)-\dis R - \diam A_k > 5 - 2\dis R - \diam I_k > 5 - 2\cdot 2\delta - 2 \ge 1.
\end{multline*}
Thus, $w_1\bigl([0, 1]\bigr)\cap R(I_n) = \emptyset$. Similarly, $w_2\bigl([0, 1]\bigr)\cap R(I_n) = \emptyset$. Hence, we have constructed a path in~$G$ between $v_k$ and~$v_l$ that does not go through~$v_n$ --- a contradiction.  
\end{proof}

Now we return to the proof of Lemma~\ref{lemma: innerRmainlemma}.

Since $\dis R < 1-2t + d - \varepsilon$ and $\diam I_n = 2t$, the inequality holds $\diam A_n < 1-2t + d - \varepsilon + 2t = d+1-\varepsilon$. 

Since $\diam X = 1$, 
for arbitrary $\varepsilon' > 0$ there exist $x,x'\in dX$ such that $|xx'| > d - \varepsilon'$. Fix $0 <\varepsilon' < \varepsilon$.

We put $A = \bigl(\inf\{t: (t, x)\in A_n\}, x\bigr)$ , $B = \bigl(\sup\{t: (t, x)\in A_n\}, x\bigr)$, $C = \bigl(\inf\{t: (t, x')\in A_n\}, x'\bigr)$, $D = \bigl(\sup\{t: (t, x')\in A_n\}, x'\bigr)$. By $x_A$, $x_B$, $x_C$, and $x_D$ denote the coordinates of projections of $A$, $B$, $C$, and~$D$, respectively, on a factor~$\R$ in~$\R_d$. Since $|AC| = |xx'| + |x_A - x_C|$, we have $$|x_A - x_C| = |AC| - |xx'|\le \diam A_n - |xx'|< d + 1 - \varepsilon - d + \varepsilon'= 1 -\varepsilon+\varepsilon'.$$ Similarly, we obtain that each of the values $|x_B-x_D|$, $|x_A-x_D|$, $|x_B-x_C|$ does not exceed~$(1 -\varepsilon+\varepsilon')$.

Without loss of generality suppose that $x_A\ge x_C$. Consider two cases. 

1) If $x_A \le x_D$, then $x_B - x_C < 1-\varepsilon + \varepsilon'$, $x_D-x_A < 1-\varepsilon+\varepsilon'$. Thus, $$|AB|+|CD| = x_B-x_C + x_D-x_A < 2-2\varepsilon+2\varepsilon'.$$ 

2) If $x_A > x_D$, then $|AB|+|CD| \le x_B-x_C < 1-\varepsilon+\varepsilon'$. 

In both cases $|AB|+|CD| < 2(1-\varepsilon+\varepsilon')$.

Without loss of generality suppose that a projection of $A_1$ on the factor~$\R$ in~$\R_d$ has smaller coordinates than a projection of~$A_n$ on the same factor.

Put $P_2 = \bigl(\sup\bigl\{t\: (t,x)\in A_1\bigr\}, x\bigr)$, $P_1 = \bigl(\sup\bigl\{t\: (t,x')\in A_1\bigr\}, x'\bigr)$, $Q_2 = \bigl(\inf\bigl\{t\: (t,x)\in A_n\bigr\}, x\bigr)$, $Q_1 = \bigl(\inf\bigl\{t\: (t,x')\in A_n\bigr\}, x'\bigr)$. 

On the one hand, since $\sup\bigl\{|x-y|\: x\in A_1,\,y\in A_n\bigr\}  = n-1+2t$, the inequality holds
\begin{align*}
\bigl||P_2Q_2| + |P_1Q_1| - 2n| \le \bigl||P_2Q_2|-n\bigr|+\bigl||P_1Q_1|-n\bigr|  \le 2\dis R + 4t+2.
\end{align*}
Therefore, 
\begin{align}\label{ineq1}
|P_2Q_2| + |P_1Q_1| \ge  2n - 2\dis R - 4t -2.
\end{align}

On the other hand, note that both segments $P_1Q_1$ and~$P_2Q_2$ lie in a union $\cup_{k=2}^{n-1} A_k$ (perhaps, excluding the points $P_1$, $Q_1$, $P_2$, $Q_2$).  

Indeed, otherwise, suppose that $P_1Q_1$ contains a point from $A_l$ such that $l < 1$. Then, similarly to Lemma~\ref{lemma: routeinG}, from the passing from $A_l$ to~$Q_1$ along s segment~$P_1Q_1$ we can construct a path in $G$ between $v_l$ and~$v_n$ that does not go through~$v_1$ --- a contradiction. 

Thus, the length of a segment~$P_1Q_1$ can be estimated from above by a sum $\sum_{t = 2}^{n-1} |L_tR_t|$, where $L_t = \bigl(\inf\{t\: (t, x)\in A_t\}, x\bigl)$, $R_t = \bigl(\sup\{t\: (t, x)\in A_t\}, x\bigl)$. Analogously, $|P_2Q_2|\le \sum_{t = 2}^{n-1} |L'_tR'_t|$, where $L'_t = \bigl(\inf\{t\: (t, x')\in A_t\}, x'\bigl)$, $R'_t = \bigl(\sup\{t\: (t, x')\in A_t\}, x'\bigl)$. Earlier, we have proved that for each $t = 2, \ldots n-1$ the inequality holds $|L_tR_t| + |L'_tR'_t|< 2-2\varepsilon + 2\varepsilon'$. Hence,
\begin{align}\label{ineq2}
|P_2Q_2| + |P_1Q_1|\le \sum_{t=2}^{n-1}\bigl(|L_tR_t| + |L'_tR'_t|\bigr)< 2(1-\varepsilon+\varepsilon')(n-2).
\end{align}

From Inequalities~(\ref{ineq1}) and~(\ref{ineq2}), it follows that $$2n - 2\dis R - 4t < 2(1-\varepsilon+\varepsilon')(n-2).$$
Dividing by~$n$, and tending~$n$ to infinity, we obtain $$1 \le 1-\varepsilon+\varepsilon' < 1,$$ a contradiction.
\end{proof}


\end{document}